\documentclass[11pt,a4paper,reqno]{article}
\author{Mara Ungureanu}
\title{Universal polynomials for counts of secant planes to projective curves}
\date{}

\AtEndDocument{\medskip
\noindent\begin{footnotesize}\textsc{Albert-Ludwigs-Universit\"at Freiburg, Mathematisches Institut, Abteilung Reine Mathematik, Ernst-Zermelo-Str 1, 79104 Freiburg}
\smallskip\\
\textit{E-mail address:} \texttt{mara.ungureanu@math.uni-freiburg.de}
\end{footnotesize}}

\usepackage[style=numeric]{biblatex}
\usepackage{amsmath, amssymb, amsthm,mathrsfs}
\usepackage{enumitem}
\usepackage{color}
\usepackage{mathtools}
\usepackage[a4paper,left=3cm,right=3cm,bottom=4cm]{geometry}
\usepackage{lmodern}
\usepackage{indentfirst}
\usepackage{tikz}
\usetikzlibrary{matrix,arrows}
\usepackage{epstopdf}
\usepackage{float}
\usepackage{microtype}

\usepackage[utf8]{inputenc}
\usepackage[T1]{fontenc}

\DeclareMathOperator{\rk}{rk}
\DeclareMathOperator{\oo}{\mathcal{O}}

\DeclareMathOperator{\p}{\mathbb{P}}

\newtheorem{thm}{Theorem}[section] 
\newtheorem{lemma}[thm]{Lemma}
\newtheorem{prop}[thm]{Proposition}
\newtheorem{cor}[thm]{Corollary}

\theoremstyle{definition}

\theoremstyle{definition}

\theoremstyle{remark}
\newtheorem{rem}{Remark}[section]

\begin{document}

\maketitle
 
\begin{abstract}
In this article we provide another method for obtaining explicit formulas yielding counts of secant planes to a projective curve.  We formulate the problem in terms of Segre classes of suitable bundles over the symmetric product of the curve and take advantage of the result of Ellingsrud, Göttsche, and Lehn concerning the structure of integrals of polynomials in Chern classes of tautological bundles over Hilbert schemes of points.
We use this to set up a recursion starting from the easy to compute case of the projective line.
\end{abstract}

\section{Introduction}
The motivation for this article is twofold.  First of all, our interest stems from a classical problem in enumerative algebraic geometry, namely that of counting the number of $e$-secant $(e-f)$-planes to a projective variety, where the parameters $e$ and $f$ are chosen such that one expects a finite answer.  In the case of smooth curves, this question has given rise to a number of well-known formulas dating back to Castelnuovo and Cayley, but has also been the object of more recent interest, as in the work of Cotterill \autocite{Co1}, \autocite{Co},  Farkas \autocite{Fa2}, and Le Barz \autocite{LB1}.
The general setup is the following: consider a general curve $C$ of genus $g$ equipped with a linear series $l$ of degree $d$ and dimension $r\geq 3$.  Let $e$ and $f$ be positive integers such that $0\leq f<e\leq d$ and denote by $C_e$ the $e$-th symmetric product of the curve.  Then
\[V_e^{e-f}(l)=\{D\in C_e \mid \dim (l-D)\geq r-e+f \}\subset C_e\]
is called the \textit{secant variety} of effective divisors of degree $e$ which parametrises the $e$-secant $(e-f-1)$-planes to the image of the curve $C$ under the embedding into $\p^r$ via $l$.  The cycle $V_e^{e-f}(l)$ is in fact a variety with a degeneracy locus structure and as such it has expected dimension
\[\exp\dim V_e^{e-f}(l) = e-f(r+1-e+f).\]
Even better, we know that if $l$ is a general linear series, then $V_e^{e-f}(l)$ is in fact equidimensional and of expected dimension \autocite{Fa2}.  We shall therefore restrict ourselves to this situation, and moreover to the case when $\dim V_e^{e-f}(l)=0$, so that the enumerative problem is well-defined.
The class of the cycle $V_e^{e-f}(l)\subset C_e$ was computed by MacDonald (see also \autocite[Chapter VIII]{ACGH}), but unfortunately the general formula is difficult to use in practice, also in our case of interest when the dimension of the secant variety is zero. 
However, when $f=1$, the number $v_e$ of $e$-secant $(e-2)$-planes to a curve of degree $d$ in $\p^{2e-2}$ has a particularly convenient formula (see \autocite{Cas} or \autocite[Chapter VIII \S 4]{ACGH}):
\[v_e = \sum_{\alpha=0}^{e}(-1)^{\alpha}\binom{g+2e-d-2}{\alpha}\binom{g}{e-\alpha},\]
and its corresponding generating function $\sum_{e\geq 0} v_e t^e$ has also been computed (see \autocite{LB} or \autocite{Co}):
\[ S(t) = \Biggl(\frac{1+\sqrt{1+4t}}{2}\Biggr)^{d}\Biggl(\frac{-1-4t+(1+2t)\sqrt{1+4t}}{2t^2}\Biggr)^{g-1}. \]
Thus, our task from this first point of view is to find similar formulas for higher values of the parameter $f$.

\bigskip

The second aspect we consider, and which will help us tackle our problem, is that of integrals of characteristic classes of tautological bundles over Hilbert schemes of points of varieties.  These objects are not only interesting because they frequently have an enumerative interpretation, as we shall see below, but also because they were crucial in establishing the fact that the cobordism class of the Hilbert scheme of points of a smooth projective surface depends only on the class of the surface itself (see \autocite{EGL}).
The general setup is as follows: consider a smooth projective variety $X$ and denote by $X^{[n]}$ the Hilbert scheme of $n$ points of $X$.  Given a vector bundle $F$ of rank $r$ on $X$ one has an associated \textit{tautological bundle} $F^{[n]}$ of rank $nr$ on $X^{[n]}$ whose fibre at $[Z]\in X^{[n]}$ is $H^0(Z,F|_Z)$.
 Our guiding result in this context is the following theorem due to Ellingsrud, Göttsche, and Lehn on the structure of such integrals:
\begin{thm}{\autocite[Proposition 4.1]{EGL}}\label{thm:egl}
Let $X$ be a smooth projective surface equipped with a vector bundle $F$  whose rank $r$ is constant on the irreducible components of $X$.  Moreover, let $P$ be a weighted homogeneous polynomial of degree $2n$ in the Chern classes of the tangent bundle $T_n$ of $X^{[n]}$ and of $F$, where each Chern class $c_i$ has weight $i$.   Then there is a universal polynomial $\widetilde{P}$ in the variables $\{\int_X c_1^2(T_X),\int_X c_2(T_X),\int_X c_1^2(F),\int_X c_2(F), \int_X c_1(T_X)c_1(F)\}$ and depending only on $P$ and $r$ such that
\begin{equation}\label{eq:eglintegral}
\deg(P\cap X^{[n]})=\int_{X^{[n]}} P=\widetilde{P}.
\end{equation} 
\end{thm}

\noindent Although the original result was stated and proved for smooth projective surfaces, one can use the same arguments as in \autocite{EGL} to obtain the analogous statement for smooth projective curves.  Furthermore, by using significantly different methods, Rennemo \autocite{Re} has extended the theorem to smooth projective varieties of any dimension.

Despite the simple form of the equality (\ref{eq:eglintegral}), it can in practice be difficult to compute the universal polynomial $\widetilde{P}$.
A particular case in which such convenient formulas can be found is that of integrals of top Segre classes over Hilbert schemes of points of curves or surfaces.  For example, if $X$ is a K3 surface with a line bundle $H$, Voisin \autocite{Vo} has recently shown that
\[\int_{X^{[n]}} s_{2n}(H^{[n]}) = 2^n \binom{\frac{H^2}{2}+2-2n}{n},\]
where $H^2$ denotes the intersection product $\int_X c_1^2(H)$. 
In fact, the formula for the generating function of the above integrals for any surface $X$ is the subject of Lehn's conjecture \autocite{L}, which was also recently established in a series of papers by Marian, Oprea, and Pandharipande (\autocite{MOP}, \autocite{MOP1}) in which the authors showed that
\begin{equation}\label{eq:0}
\sum_{n=0}^{\infty}z^n \int_{X^{[n]}}s_{2n}(H^{[n]}) = \frac{(1-w)^a(1-2w)^b}{(1-6w+6w^2)^c},
\end{equation}
for constants $a,b,c$ depending only on $H^2$, $K_X^2$, $H.K_X$ and $\chi(\mathcal{O}_X)$ in a precise way, and with a change of variables $z \leftrightarrow w$ (for details see for example equation (2) of \autocite{MOP1}).
In the case of curves, generating functions for the analogous top Segre integral of tautological bundles associated to vector bundles of rank $r\geq 1$ have been given in \autocite{MOP1} and \autocite{Wa}.

Thus, in this context, our task is to find universal polynomials for integrals of other weighted homogeneous polynomials of degree $n$ (resp.~$2n$) in Segre classes of tautological bundles over $X^{[n]}$, where $X$ is a curve (resp.~a surface).

\bigskip

The connection between the two aforementioned tasks is realised via the enumerative interpretation of the Segre integrals.  More precisely, the top Segre integrals yield the expected number of $n$-secant $(n-2)$-planes to the curve (resp.~surface) $X$ in $\p^{2n-2}$ (resp.~$\p^{3n-2}$), i.e.~precisely the counts discussed above in the case of curves when $f=1$.  
It is therefore natural to consider the next case, namely $f=2$, and investigate to what extent the methods developed for computing Segre integrals may help in finding explicit enumerative formulas.  

\bigskip

Hence, the aim of this article is to use these methods to compute the number of $e$-secant $(e-3)$-planes to a smooth curve $C$ in $\p^{\frac{3e}{2}-3}$, or equivalently, as shall be explained in the course of the paper, the universal polynomial corresponding to
\begin{equation}\label{eq:1}
 \int_{C^{[e]}} s_{\frac{e}{2}}(L^{[e]}) - s_{\frac{e}{2}-1}(L^{[e]})s_{\frac{e}{2}+1}(L^{[e]}),
\end{equation}
where $e\geq 4$ is an even integer and $L$ is the line bundle whose sections give the embedding of $C$ in projective space.  For simplicity we restrict to complete linear systems, so we have $h^0(C,L)=\frac{3e}{2}-2$ and we write the secant variety as $V_e^{e-f}(L)$, where $f=2$.  Note that the parameter choice indeed yields $\exp\dim V_e^{e-f}(L)=0$.

\section*{Roadmap}
The paper is organised as follows:
\begin{itemize}
	\item In Section \ref{sec:secantrecap} we summarise the description of secant varieties as degeneracy loci of maps between vector bundles on the symmetric product of the curve.  We then recall how to use it to obtain the expression of the class of $V_e^{e-f}(L)$ in terms of Segre classes of $L^{[e]}$ in the cases $f=1$ and $f=2$, thus recovering the top Segre integral and (\ref{eq:1}), respectively.
	\item The main results of Section \ref{sec:structuretheorem}, namely Theorem \ref{thm:structure} and Corollary \ref{cor:struct}, are inspired by the Structure Theorem 4.2 of \autocite{EGL}, which, in a simplified formulation, states the following: consider a multiplicative characteristic class $\phi$, i.e.~with the property that $\phi(F\oplus F')=\phi(F)\phi(F')$, for two vector bundles $F$ and $F'$ on a variety $X$.  Then, if $X$ is a smooth projective surface with a vector bundle $F$, the authors prove that
	\[\sum_{n=0}^{\infty} \int_{X^{[n]}} \phi(F^{[n]})z^n = \exp\Bigl(\sum_{i=1}^5 C_i A_i(z)\Bigr),\]
	where the $A_i(z)$ are power series in $\mathbb{Q}[[z]]$ and the $C_i$ denote the entries of the vector $(\int_X c_1^2(T_X),\int_X c_2(T_X),\int_X c_1^2(F),\int_X c_2(F), \int_X c_1(T_X)c_1(F))$.  The analogous statement for curves also holds.  The total Segre and Chern classes are such  multiplicative classes, and this property, together with the Structure Theorem above have been crucial ingredients in finding the formulas for integrals of top Segre classes and hence the expression (\ref{eq:0}).  However, for $f=2$ we shall be dealing instead with products of additive classes.  Thus, our Theorem \ref{thm:structure} and Corollary \ref{cor:struct} provide useful formulas for dealing with generating functions of integrals of products of additive classes of tautological bundles. 
	
	\item Section \ref{sec:prelim} is dedicated to the computation of the integral (\ref{eq:1}).  To do so we use the results of Section \ref{sec:structuretheorem} to establish recursions starting from the easy case of the projective line. After first establishing some preliminary facts in Section \ref{sec:reduction} and Section \ref{sec:p1},  this is done in two parts:
	\item In Section \ref{sec:lucky} we first consider a special case of our problem, namely $f=2$ and $e=4$.  A closed expression for this count already exists (see \autocite[Chapter VIII \S 4, Example 3]{ACGH}), but we provide here a different way to approach the computation.  More precisely, we notice that for any vector bundle $F$ the following holds:
	\[ s_2^2(F) - s_1(F)s_3(F) = c_4(F)+s_4(F)-6c_1(F)ch_3(F). \]
Now, the integrals $\int_{C^{[4]}}c_4(L^{[4]})$ and $\int_{C^{[4]}}s_4(L^{[4]})$ can be immediately read off from the literature (see \autocite{MOP1} or \autocite{Wa}), where, as we mentioned before, one takes advantage of the multiplicativity of the total Chern and Segre classes. 
Thus, to find (\ref{eq:1}) it is enough to compute $\int_{C^{[4]}}c_1(L^{[4]})ch_3(L^{[4]})$ using the additivity of $c_1$ and of the Chern character and the results of Section \ref{sec:structuretheorem}.  This yields the formula:
	\[\int_{C^{[4]}}c_1(L^{[4]})ch_3(L^{[4]}) = \frac{1}{6}((d+g-3)^2-4g-g(g-1)),\]
	for a smooth curve $C$ of genus $g$ with $L$ of degree $d\geq 4$.
	\item Finally, in Section \ref{sec:generalcasecurves} we present a structure result for the general case of (\ref{eq:1}).
\end{itemize}

\section{Classes of secant varieties}\label{sec:secantrecap}
We begin by recalling a few well-known facts regarding secant varieties to projective curves or surfaces, and in particular their description as degeneracy loci of vector bundle morphisms over Hilbert schemes of points.
We follow Fulton \autocite{Ful}.  Let $X$ be a variety, $\varphi:E\rightarrow F$ a morphism of vector bundles over $X$, and $k\leq \min{\rk E,\rk F}$ a positive integer.  Denote by $A_m(X)$ the  $m$-th Chow group of $X$.
The $k$-th degeneracy locus
\[ D_k(\varphi) = \{ x\in X \mid \rk (\varphi)\leq k  \} \]
has codimension at most $(\rk E-k)(\rk F-k)$ in $X$, if non-empty.    Now set 
\[m=\dim X - (\rk E-k)(\rk F-k)\] and denote by $\mathbb{D}_k(\varphi)$ the class of $D_k(\varphi)$ in $A_m(D_k(\varphi))$.  By the Thom-Porteous formula, the image of $\mathbb{D}_k(\varphi)$ inside $A_m(X)$ is given by 
\begin{equation}\label{eq:classdef}
\Delta^{(\rk E - k)}_{\rk F - k}(c(F-E))\cap [X],
\end{equation} 
where $\Delta_q^{(p)}(c)$ denotes the determinant of the $p\times p$ matrix with entries $(c_{q+j-i})_{1\leq i,j \leq p}$.

We are interested in the situation where the variety in question is the Hilbert scheme of points $X^{[e]}$, where $X$ is either a smooth projective curve or a smooth projective surface.  In both cases $X^{[e]}$ is a smooth projective variety of dimension $e$ if $X$ is a curve and $2e$ if $X$ is a surface.  Moreover, if $X$ is a curve we also have that $X^{[e]}$ is isomorphic to the symmetric product $X_e$.  Moreover, let $L$ be a line bundle of degree $d$ on $X$ and with $h^0(X,L)=r+1$ whose sections give an embedding of $X$ in $\p^r$.

Secant varieties $V_e^{e-f}(L)$ of $e$-secant $(e-f-1)$-planes to $X$ can be endowed with a degeneracy locus structure inside the Hilbert scheme $X^{[e]}$ of $0$-dimensional subschemes of $X$ of length $e$.  To see this, let $E=\oo_{X^{[e]}}\otimes H^0(X,L)$ be the trivial vector bundle of rank $r+1$ on $X^{[e]}$ and $L^{[e]}:=\tau_*(\sigma^* L\otimes\oo_{\mathcal{U}})$ be the tautological bundle of rank $e$, where $\mathcal{U}$ is the universal family of subschemes parametrised by $X^{[e]}$
\[ \mathcal{U} = \{ (x,	[Z]) \mid x\in Z\text{ and }[Z]\in X^{[e]} \} \subset X\times X^{[e]}, \]
and $\sigma$, $\tau$ are the usual projections: 
 \begin{figure}[H]\centering
  \begin{tikzpicture}
    \matrix (m) [matrix of math nodes,row sep=2em,column sep=1em,minimum width=1em]
  {
     & X \times X^{[e]} & \supset \mathcal{U} \\
     X & & X^{[e]}\\};
  \path[-stealth]
    (m-1-2) edge node [auto,swap] {$\sigma$} (m-2-1)
            edge node [auto]{$\tau$}  (m-2-3);
 \end{tikzpicture}
 \end{figure}
 
Let $\Phi:E\rightarrow L^{[e]}$ be the vector bundle morphism obtained by pushing down to $X^{[e]}$ the restriction map $\sigma^* L \rightarrow \sigma^* L \otimes\oo_{\mathcal{U}}$.  Then one obtains $V_e^{e-f}(L)$ as the $(e-f)$-th degeneracy locus of $\Phi$.  Indeed, fibrewise, the morphism $\Phi$ is given by the restriction
\[ \Phi_Z : H^0(X,L) \rightarrow H^0(X,L|_Z), \]
while $Z\in V_e^{e-f}(L)$ if and only if $\rk \Phi_Z \leq e-f$.

From the above considerations we see that $V_e^{e-f}(L)$ has expected dimension
\[ \exp\dim V_e^{e-f}(L) = e\cdot\dim X  - f(r+1-e+f). \]
Our enumerative study thus restricts itself to the case $\exp\dim V_e^{e-f}(L) = 0$ where we expect to have only finitely many $e$-secant $(e-f-1)$-planes to $X$.  Of course, the numbers obtained from (\ref{eq:classdef}) are virtual, unless one is able to prove that the secant variety $V_e^{e-f}(l)$ is indeed not empty and of expected dimension 0.  This has been verified for general curves with general linear series by Farkas \autocite{Fa2}.  

In the remainder of this section we describe the class of $V_e^{e-f}(L)$ from (\ref{eq:classdef}) in more detail in the cases $f=1$ and $f=2$.  As mentioned in the Introduction, the case $f=1$ is by now well-known for both curves and surfaces, while the case $f=2$ constitutes the main focus of this article.

\subsection{The case $f=1$}
If $f=1$, then we have that $e\cdot\dim X = r-e+2$, so that if $X$ is a curve, then $r=2e-2$ and if $X$ is a surface, then $r=3e-2$.  From now on we write $C$ for curves and $S$ for surfaces.  Hence $\rk E = r+1 \in \{ 2e-1, 3e-1 \}$ and $\rk L^{[e]}=e$, which means that the class of $V_e^{e-f}(L)$ is given via (\ref{eq:classdef}) by
\[ \Delta^{(e)}_1 (c(L^{[e]}))\cap C^{[e]}=\left|\begin{array}{cccc}
c_1(L^{[e]}) & c_2(L^{[e]}) & \cdots & c_e(L^{[e]})\\
1 & c_1(L^{[e]}) & \cdots & c_{e-1}(L^{[e]})\\
\vdots & \vdots & \vdots & \vdots \\
0 & 0 & 1 & c_1(L^{[e]})
\end{array} \right| \in A_0(C^{[e]}) \]
for curves and 
\[ \Delta^{(2e)}_1 (c(L^{[e]}))\cap S^{[e]}=\left|\begin{array}{cccc}
c_1(L^{[e]}) & c_2(L^{[e]}) & \cdots & c_{2e}(L^{[e]})\\
1 & c_1(L^{[e]}) & \cdots & c_{2e-1}(L^{[e]})\\
\vdots & \vdots & \vdots & \vdots \\
0 & 0 & 1 & c_1(L^{[e]})
\end{array} \right| \in A_0(S^{[e]}) \]
for surfaces. In both cases we used the fact that $E$ is a trivial bundle so that $c(L^{[e]}-E)=\frac{c(L^{[e]})}{c(E)}=c(L^{[e]})$.  Although the above expressions for the degeneracy loci $V_e^{e-f}(L)$ appear quite complicated, one may obtain a certain degree of simplification by rewriting them in terms of determinants in Segre classes of the dual bundle $L^{[e]}$.
To see how, we record here the following useful result, which is a direct consequence of Lemma 14.5.1 of \autocite{Ful}:
\begin{lemma}
Let $\alpha_1,\alpha_2,\ldots$ and $\beta_1,\beta_2,\ldots$ be two sequences of commuting variables related by the identity:
\[ (1+\alpha_1 t + \alpha_2 t^2 + \cdot)\cdot (1-\beta_1 t + \beta_2 t^2 - \beta_3 t^3 + \cdots)  =1. \]
Then $\Delta_q^{(p)} (\alpha) = \Delta _p^{(q)}(\beta)$ for any positive integers $p$ and $q$.
\end{lemma}
Substituting $c_i(L^{[e]})$ for the $\alpha_i$ and $s_i(L^{[e]})$ for the $\beta_i$ yields:
\begin{equation}\label{eq:classcurve}
\Delta^{(e)}_1 (c(L^{[e]}))\cap C^{[e]} = s_e((L^{[e]})^{\vee})\cap [C^{[e]}]
\end{equation}   for curves and
\begin{equation}\label{eq:classsurface}
\Delta^{(2e)}_1 (c(L^{[e]}))\cap S^{[e]} = s_{2e}((L^{[e]})^{\vee})\cap [S^{[e]}]=s_{2e}(L^{[e]})\cap [S^{[e]}]
\end{equation} for surfaces.

Thus the number of $e$-secant $(e-2)$-planes to a curve in $\p^{2e-2}$ and to a surface in $\p^{3e-2}$ is given by 
\[ \int_{C^{[e]}} s_e((L^{[e]})^{\vee}) = \deg (s_e((L^{[e]})^{\vee})\cap [C^{[e]}]),  \]
and
\[ \int_{S^{[e]}} s_{2e}((L^{[e]})^{\vee}) = \deg(s_{2e}(L^{[e]})\cap [S^{[e]}]),  \]
respectively.

These counts (or their corresponding generating functions) were indepedently computed by Cotterill \autocite{Co}, Le Barz \autocite{LB}, and Wang \autocite{Wa} for curves, whereas the surfaces case was settled in a series of papers by Marian, Oprea, and Pandharipande (\autocite{MOP}, \autocite{MOP1}), Voisin \autocite{Vo}, with some of the earliest partial results dating back to Lehn \autocite{L}.  In what follows we are interested in carrying out the computations for curves in the case $f=2$.

\subsection{The case $f=2$}
If $f=2$ we have that $e\cdot \dim C = 2(r-e+3)$ so that $\rk E=r+1= \frac{3e}{2}-2$.  The rank of $L^{[e]}$ is, as before, $e$ so using (\ref{eq:classdef}) again we can write the class of $V_e^{e-f}(L)$ in this case as
\[ \Delta^{(e/2)}_2(c(L^{[e]})) \cap C^{[e]} = \left| \begin{array}{cccc}
c_{2}(L^{[e]}) & c_{3}(L^{[e]}) & \cdots & c_{\frac{e}{2}+1}(L^{[e]})\\
c_1(L^{[e]}) & c_2(L^{[e]}) & \cdots & c_{\frac{e}{2}}(L^{[e]})\\
\vdots & \vdots & \vdots & \vdots\\
0 & \cdots & c_1(L^{[e]}) & c_{2}(L^{[e]})
\end{array} \right|.  \] 
To write this in terms of Segre classes we once more use the framework of Section \ref{sec:prelim}.  We have the partition $\lambda=(2,2,\ldots,2)$ of $e$ corresponding to $\Delta^{(e/2)}_2(c(L^{[e]}))$, whose conjugate partition is $\mu=\bigl(\frac{e}{2},\frac{e}{2}\bigr)$.  Thus
\begin{align*}
\Delta^{(e/2)}_2(c(L^{[e]})) &= \Delta^{(2)}_{e/2}(s((L^{[e]})^{\vee}))=\left|\begin{array}{cc}
s_{\frac{e}{2}}((L^{[e]})^{\vee}) & s_{\frac{e}{2}+1}((L^{[e]})^{\vee})\\
s_{\frac{e}{2}-1}((L^{[e]})^{\vee}) & s_{\frac{e}{2}}((L^{[e]})^{\vee}) 
\end{array} \right|\\ &= s_{\frac{e}{2}}(L^{[e]})^2 - s_{\frac{e}{2}-1}(L^{[e]})s_{\frac{e}{2}+1}(L^{[e]}).
\end{align*}

Hence the class of $V_e^{e-f}(L)$ can be written in the case $f=2$ in terms of a second degree polynomial in certain Segre classes.  This is already a simplification compared to the original expression although less convenient than the top Segre class that appears in the $f=1$ case.  In what follows we compute the integrals
\begin{equation}\label{eq:segrecurve}
\int_{C^{[e]}} s_{\frac{e}{2}}(L^{[e]})^2 - s_{\frac{e}{2}-1}(L^{[e]})s_{\frac{e}{2}+1}(L^{[e]}).
\end{equation} 
We shall use an inductive argument starting from the simple case of $\p^1$ equipped with the Serre twisting sheaf.

\begin{rem}
For surfaces $S$ we have
\[ \Delta^{(e)}_2(c(L^{[e]}))\cap S^{[e]} = \left| \begin{array}{cccc}
c_2(L^{[e]}) & c_3(L^{[e]}) & \cdots & c_{e+1}(L^{[e]})\\
c_1(L^{[e]}) & c_2(L^{[e]}) & \cdots & c_e(L^{[e]}) \\
\vdots & \vdots & \vdots & \vdots \\
0 & \cdots & c_1(L^{[e]}) & c_2(L^{[e]})
\end{array} \right|\] 
and the same argument as for curves yields
\[\Delta^{(e)}_2(c(L^{[e]}))=s_e(L^{[e]})^2 - s_{e-1}(L^{[e]})s_{e+1}(L^{[e]}).\]
\end{rem}

\section{Structure theorems for generating functions of tautological integrals}\label{sec:structuretheorem}
A very useful tool in computing integrals of certain characteristic classes is the beautiful Structure Theorem 4.2 of \autocite{EGL}, which we now recall in a slightly simplified form that suffices for our needs.  In keeping with the notation in loc.cit., let $K(X)$ be the Grothendieck group generated by locally free sheaves on $X$.
Moreover, let $\psi:K(X)\rightarrow H^*(X,\mathbb{Q})$ be a multiplicative function, i.e.~a group homomorphism from the additive group $K(X)$ of a smooth projective curve or surface $X$ into the multiplicative group of units in $H^*(X,\mathbb{Q})$.  Examples of such multiplicative functions are the total Chern or the total Segre classes. 
Now, given $x\in K(X)$, one defines a power series in $\mathbb{Q}[[z]]$ as follows:
\[H_{\psi}(X,x):=\sum_{k=0}^{\infty}\int_{X^{[k]}}\psi(x^{[k]})z^k.\]
We then have:

\begin{thm}(\autocite[Theorem 4.2]{EGL})\label{thm:structegl}
For each integer $r$ there are universal power series $A_i\in \mathbb{Q}[[z]]$ with $i=1,2$ if $X$ is a curve and $i=1,\ldots,5$ if $X$ is a surface, depending only on $\psi$ and $r$, such that for each $x\in K(X)$ of rank $r$ on every component of $X$ one has
\[H_{\psi}(X,x) = e^{ \int_X c_1(x)A_1 + \int_X c_1(T_X)A_2}\]
when $X$ is a curve and
\[H_{\psi}(X,x) = e^{\int_X c_1^2(x)A_1 + \int_X c_2(x)A_2 + \int_X c_1(x)c_1(T_X)A_3 + \int_X c_1^2(T_X)A_4 + \int_X c_2(T_X)A_5}\]
if $X$ is a surface.
\end{thm}

\begin{rem}\label{rem1}
We note here that in \autocite{EGL}, the proof of Theorem \ref{thm:structegl} is done only for surfaces.  However, the same argument can be applied in the case of curves using the elements $(\p^1,r\cdot 1)$ and $(\p^1, \oo_{\p^1}(1)+(r-1)\cdot 1)$ of the space $\mathcal{K}_r=\{(X,x)\mid X \text{ algebraic curve, }x\in K(X)\}$, which under the map $\gamma:\mathcal{K}_r\rightarrow\mathbb{Q}^2$ sending $(X,x)$ to $(\int_X c_1(x),\int_X c_1(T_X))$ yield two linearly independent vectors $(0,2)$ and $(1,2)$.
\end{rem}

Unfortunately, classes of various geometric cycles inside Hilbert schemes of points of curves or surfaces are not always multiplicative.  While the multiplicativity of the total Segre class together with Theorem \ref{thm:structegl} was crucial in determining $\int_{X^{[e]}}s_{top}$ and hence the number of $e$-secant $(e-2)$-planes to a surface in $\p^{3e-2}$, unfortunately for different choices of parameters this is no longer the case.  Indeed, if we take $f=2$ instead of $f=1$, then  the cycle classes in (\ref{eq:classcurve}) and (\ref{eq:classsurface}) corresponding to $e$-secant $(e-3)$-planes to curves or surfaces are no longer multiplicative and other methods are therefore required to compute them.  In fact, in certain cases, these classes can be written down as a product of two additive classes.  Hence, it is useful for us to understand instead the behaviour of the power series
\[ H_{\phi\psi}(X,x):=\sum_{k=0}^{\infty}\int_{X^{[k]}}\phi(x^{[k]})\psi(x^{[k]})z^k, \]
where $\phi,\psi:K(X)\rightarrow H^*(X,\mathbb{Q})$ are additive functions, or more precisely, group homomorphisms between the additive groups $K(X)$ and $H^*(X,\mathbb{Q})$.   In particular, we are interested in how the power series $H_{\phi\psi}$ relates to the two power series
\begin{align*}
H_{\phi}(X,x)&:=\sum_{k=0}^{\infty}\int_{X^{[k]}} \phi(x^{[k]})z^k\text{ and }\\
H_{\psi}(X,x)&:=\sum_{k=0}^{\infty}\int_{X^{[k]}} \psi(x^{[k]})z^k.
\end{align*} 
To this end we prove:

\begin{thm}\label{thm:structure}
For each integer $r$ there are universal power series $B_i\in\mathbb{Q}[[z]]$ with $i=1,2$ if $X$ is a curve and $i=1,\ldots,5$ if $X$ is a surface, depending only on $r$, $\phi$, and $\psi$, such that for each $x\in K(X)$ of rank $r$ on every component of $X$ one has 
\begin{equation}\label{eq:linear1}
H_{\phi\psi}(X,x) - H_{\phi}(X,x)H_{\psi}(X,x) =\int_{X}c_1(x) B_1 +\int_X c_1(T_X) B_2
\end{equation}  
for curves and
\begin{align}\label{eq:linear2}
H_{\phi\psi}(X,x) - H_{\phi}(X,x)H_{\psi}(X,x) &= \int_X c_1^2(x)B_1 + \int_X c_2(x)B_2\\
 &+ \int_X c_1(x)c_1(T_X)B_3 + \int_X c_1^2(T_X)B_4 + \int_X c_2(T_X)B_5\nonumber
\end{align}  
for surfaces.
\end{thm}
\begin{proof}
As in the proof of Theorem \ref{thm:structegl} in \autocite{EGL} and as already hinted at in Remark \ref{rem1}, we start by considering the space $\mathcal{K}_r=\{(X,x)\mid X \text{ algebraic curve or surface, }x\in K(X)\}$ and the map $\gamma:\mathcal{K}_r\rightarrow\mathbb{Q}^j$, with $j=2$ if $X$ is a curve and $j=5$ if $X$ is a surface, defined by
\begin{equation}\label{eq:vectorcurve}
(X,x)\mapsto \left(\int_X c_1(x), \int_X c_1(T_X)\right)
\end{equation}
in the case of curves and by
\begin{equation}\label{eq:vectorsurface}
(X,x)\mapsto \left( \int_X c_1^2(x),\int_X c_2(x),\int_X c_1(x)c_1(T_X),\int_X c_1^2(T_X),\int_X c_2(T_X) \right)
\end{equation}
in the case of surfaces.  As shown in Remark \ref{rem1} and \autocite{EGL}, in both cases one can find elements of $\mathcal{K}_r$ whose image under $\gamma$ form a basis of $\mathbb{Q}^j$, for the appropriate value of $j$.

Now, if $X=X_1\sqcup X_2$ and we set  $x_i:=x|_{X_1}$, then $\gamma(X,x) = \gamma(X_1,x_1) + \gamma(X_2,x_2)$ by the additivity of the entries of the vectors in (\ref{eq:vectorcurve}) and (\ref{eq:vectorsurface}) under disjoint unions.
Furthermore, the Hilbert scheme of a disjoint union satisfies $X^{[k]}=\bigsqcup_{k_1+k_2=k}X_1^{[k_1]}\times X_2^{[k_2]}$, and the class of the corresponding tautological sheaf can be decomposed as 
\[x^{[k]}|_{X_1^{[k_1]}\times X_2^{[k_2]}}=p_1^* x_1^{[k_1]}\oplus p_2^* x_2^{[k_2]},\]
where the $p_i:X_1^{[k_1]}\times X_2^{[k_2]}\rightarrow X_i^{[k_i]}$ with $i=1,2$ are the usual projections.
From the additivity of $\phi$ and $\psi$ we obtain
\[H_{\phi\psi}(X,x) = H_{\phi\psi}(X_1,x_1) + H_{\phi\psi}(X_2,x_2) + H_{\phi}(X_1,x_1)H_{\psi}(X_2,x_2) + H_{\phi}(X_2,x_2)H_{\psi}(X_1,x_1).\]
We now reinterpret the power series $H_{\phi\psi}(X,x)$, $H_{\phi}(X,x)$, and $H_{\psi}(X,x)$ as three functions $H_{\phi\psi}$, $H_{\phi}$, $H_{\psi}:\mathcal{K}_r\rightarrow\mathbb{Q}[[z]]$. We then deduce from Theorem \ref{thm:egl} that these functions all factor through $\gamma$ and three  additive maps $h,f,g:\mathbb{Q}^j\rightarrow\mathbb{Q}[[z]]$, respectively.  By  construction we see that $f$ and $g$ are also additive functions.
Since the image of $\gamma$ is Zariski dense in $\mathbb{Q}^j$, we conclude that 
\begin{align*}
h(y_1+y_2) &= h(y_1)+h(y_2)+f(y_1)g(y_2)+f(y_2)g(y_1),\\
(fg)(y_1+y_2) &= f(y_1)g(y_1)+f(y_2)g(y_2)+f(y_1)g(y_2)+f(y_2)g(y_1),
\end{align*}
for all $y_1,y_2\in\mathbb{Q}^j$, with $j=2$ when $X$ is a curve and $j=5$ when $X$ is a surface.
Hence $(h-fg)(y_1+y_2)=(h-fg)(y_1)+(h-fg)(y_2)$, i.e.~$h-fg$ is a linear function, which proves the theorem.
\end{proof}

It is not difficult to extend the above result to the product of several additive functions $\phi_1,\ldots\phi_n:K(X)\rightarrow H^*(X,\mathbb{Q})$.  To do so, we introduce the following power series: 
\[ H_{\phi_{i_1\ldots\phi_{i_n}}}:= \sum_{k=0}^{\infty}\int_{X^{[k]}}\phi_{i_1}(x^{[k]})\cdots\phi_{i_n}(x^{[k]})z^k.\]
Using induction on $n$ with base case given by the statement of Theorem \ref{thm:structure}, one gets 

\begin{cor}\label{cor:struct}
For all integers $r$ and $n\geq 2$ there are universal power series $B_i\in\mathbb{Q}[[z]]$ with $i=1,2$ if $X$ is a curve and $i=1,\ldots,5$ if $X$ is a surface, depending only on $r$, $\phi_{1},\ldots,\phi_n$ such that for each $x\in K(X)$ of rank $r$ on every component of $X$ one has 
that
\[H_{\phi_i\ldots\phi_n}-H_{\phi_1}H_{\phi_2\ldots\phi_n}-\cdots-(n-1)H_{\phi_1}H_{\phi_2}\cdots H_{\phi_n}\]
is linear in the $B_i$, i.e.~it is equal to the right-hand-side of (\ref{eq:linear1}) or (\ref{eq:linear2}) if $X$ is a curve or a surface, respectively.
\end{cor}
\begin{proof}
The proof is essentially the same as the one for Theorem \ref{thm:structure}.  The idea is once more to view the power series $H_{\phi_{i_1}\ldots\phi_{i_l}}$ with $i_1,\ldots,i_l \in \{1,\ldots,n\}$ and $1\leq l\leq n$ as functions $\mathcal{K}_r \rightarrow \mathbb{Q}[[z]]$ that each factor through the map $\gamma:\mathcal{K}_r \rightarrow\mathbb{Q}^j$ from above and a corresponding map $f_{i_1\ldots i_l}:\mathbb{Q}^j\rightarrow\mathbb{Q}[[z]]$, with $j=2$ if $X$ is a curve and $j=5$ if $X$ is a surface.  Consider the map $F:\mathbb{Q}^j \rightarrow \mathbb{Q}[[z]]$ defined by
\begin{align*}
F(y) &:= f_{1\ldots n}(y) - \sum_{i=1}^n f_i(y) f_{1\ldots\hat{i}\ldots n}(y) - \sum_{i,j=1}^n f_i(y) f_j(y) f_{1\ldots\hat{i}\ldots\hat{j}\ldots n}(y) - \cdots \\ &- (n-1)f_1(y) f_2(y)\cdots f_n(y),
\end{align*} 
for all $y\in \mathbb{Q}^j$ and where $\hat{i}$ indicates the removal of the index $i$.  Using induction on $n$ and the fact that $\gamma$ has a dense image in $\mathbb{Q}^j$ we get
$F(y_1+y_2)=F(y_1)+F(Y_2)$ for all $y_1,y_2 \in \mathbb{Q}^j$.
\end{proof}

Thus, the remainder of the paper is dedicated to applications of Theorem \ref{thm:structure} and Corrolary \ref{cor:struct} to the problem of counting the number of certain secant planes to algebraic curves.  

\section{Computing the Segre integrals}\label{sec:prelim}
This section is dedicated to the calculation of the Segre integral
\begin{equation}\label{eq:classcurve1}
\int_{C^{[e]}} s_{\frac{e}{2}}(L^{[e]})^2 - s_{\frac{e}{2}-1}(L^{[e]})s_{\frac{e}{2}+1}(L^{[e]}),
\end{equation}
where $C$ is a smooth curve with a line bundle $L$, and $e\geq 4$ is an integer chosen such that $\exp\dim V_{e}^{e-2}(L)=0$.
The general strategy we follow here is that of reducing the computation to cases that are easily understood. 
One of the main ingredients of the reduction is Theorem \ref{thm:egl} that allows us to parametrise the sought-after integrals by a small number of characteristic numbers that are well-behaved with respect to disjoint unions.  We begin by describing our reduction procedure in general terms in Section \ref{sec:reduction}.  Sections \ref{sec:p1} and \ref{sec:lucky} are dedicated to implementing this reduction to a particular choice of parameters yielding a very nice form of the above integral, but which unfortunately is not generalisable.  In Section \ref{sec:generalcasecurves} we give a structure formula for (\ref{eq:classcurve1}).

\subsection{Reduction to the projective line}\label{sec:reduction}
As mentioned before, for a curve $C$ of genus $g$ equipped with a line bundle $L$ of degree $d$ it follows from Theorem \ref{thm:egl} that the integral $P_{g,d}=\int_{C^{[k]}}P$, where $P$ is a polynomial in the Chern classes of $L^{[k]}$, is itself a polynomial in $\int_C c_1(C)=\chi(C)=2-2g$ and $\int_C c_1(L) = \deg(L)$.
Thus in order to compute $P_{g,d}$ we attempt to find, or construct a suitable pair $(C,L)$ with the correct characteristic numbers $g$ and $d$ for which the calculation is easy; for characteristic numbers $g'$ and $d'$ for which such a convenient pair $(C',L')$ cannot easily be found,
we instead relate the pair $(C',L')$ to another pair $(C'',L'')$ with $g(C'')=g'$ and $\deg(L'')=d'$ via a geometric construction and use this to obtain $P_{g', d'}$.  We now explain this procedure in more detail.

For the first step we compute $P_{0,d}$ for any polynomial $P$ in the Chern classes of $L^{[k]}$ using the very convenient pair $(\p^1,\oo_{\p^1}(d))$.  This is done in Section \ref{sec:p1}.  With this result in hand, we then proceed as follows:
fix two positive integers $g$ and $d$.  Let $C_1=C\sqcup C_2$ be the disjoint union of a curve $C$ of genus $g$ of with another curve $C_2$ of genus $g_2$.  Equip $C_1$ with a line bundle $L_1$ so that $L= (L_1)|_{C_1}$ has degree $d$ and set  $L_2=L|_{C_2}$.  Suppose $L_1$ and $L_2$ are line bundles of degree $d_1$ and $d_2$, respectively.  As mentioned before, the Hilbert scheme of $k$ points of $C_1$ is given by $C_1^{[k]} = \bigsqcup_{k_1 + k_2  = k} C^{[k_1]}\times C_2^{[k_2]}$, while the tautological bundle corresponding to $L$ has the following property:
\[ (L_1)^{[k]}|_{C^{[k_1]}\times C_2^{[k_2]}} = p_1^* L^{[k_1]} \oplus p_2^* L_2^{[k_2]},\]
where the $p_i$ denotes the projection to $C_i^{[k_i]}$ for $i=1,2$.
Furthemore, one immediately sees that since $\chi(C_1)=\chi(C)+\chi(C_2)$, and if $g_1$ denotes the genus of $C_1$, then
\[g_1=g+g_2-1\]
must also hold.  In addition, $\deg(L_1)=d + \deg(L_2)$.

Thus, if we take $C_2\simeq \p^1$ and $L_2=\oo_{\p^1}(1)$, then $C_1$ has genus $g-1$ and the line bundle $L_1$ is of degree $\deg(L_1)=d+1$.  Hence, one can compute $P_{g,d}$ by induction on $g$ with base case $P_{0,d}$ and induction step going from $P_{g-1,d+1}$ to $P_{g,d}$.  One of course needs to understand for each choice of $P$ how $P_{g-1,d+1}$ and $P_{g,d}$ relate to each other.  
We shall use this strategy in Section \ref{sec:lucky} in order to calculate a special case of the integral in (\ref{eq:classcurve1}).

\subsection{Characteristic classes of tautological bundles on $\p^1$}\label{sec:p1}
We now compute the Chern and Segre classes of the tautological bundle associated to $\oo_{\p^1}(d)$.  Note first that we have the following concrete description of the Hilbert scheme of points of the projective line: $(\p^1)^{[k]}\simeq\p^k$.  From the proof of Theorem 2 in \autocite{MOP1} we have that
\begin{equation}\label{eq:chernchartaut}
ch \Bigl(\oo_{\p^k}(d)^{[k]}\Bigr) = (d+1) - (d-k+1)\cdot \exp(-h),
\end{equation} 
where $h$ denotes the hyperplane class on $\p^k$.  We therefore have that $ch_0\Bigl(\oo_{\p^k}(d)^{[k]}\Bigr) = \rk\Bigl(\oo_{\p^k}(d)^{[k]}\Bigr)=k$ and for $i\geq 1$, 
\[ ch_i \Bigl(\oo_{\p^k}(d)^{[k]}\Bigr) =  \frac{(d-k+1)}{i!}h^i=\frac{1}{i!}p_i,\]
where $p_i$ is the $i$-th power sum symmetric polynomial in the Chern roots of $\oo_{\p^k}(d)^{[k]}$, i.e.
\[ p_i = \alpha_1^i + \cdots + \alpha_k^i. \]
To find the Chern classes of $\oo_{\p^k}(d)^{[k]}$, recall that the following expression of the total Chern class in terms of elementary symmetric polynomials in the Chern roots:
\[ c\Bigl(\oo_{\p^k}(d)^{[k]}\Bigr) = 1 + e_1(\alpha_1,\ldots,\alpha_k) + \cdots + e_k(\alpha_1,\ldots,\alpha_k). \]
Moreover, we also have the following relation between the $e_j$ and the $p_j$:
\begin{equation}\label{eq:symmpolys}
e_l = \frac{1}{l}\sum_{j=1}^l (-1)^{j-1}e_{l-j}p_j,
\end{equation}
where we omit the argument $(\alpha_1,\ldots,\alpha_k)$ to ease notation.  Using this, we can prove 
\begin{lemma}\label{lemma}
For all positive integers $l$, we have:
\begin{align*}
c_l\Bigl(\oo_{\p^k}(d)^{[k]}\Bigr) &= \binom{d-k+l}{l}h^l,\\
s_l\Bigl(\oo_{\p^k}(d)^{[k]}\Bigr) &= (-1)^l\binom{d-k+1}{l}h^l.
\end{align*}
\end{lemma}
\begin{proof}
We show the computation for the Chern classes.  The Segre classes are then obtained in a similar fashion.
The proof goes by induction.  The base case is confirmed immediately:
\[ c_1\Bigl(\oo_{\p^k}(d)^{[k]}\Bigr) = (d-k+1)h, \]
as expected from (\ref{eq:chernchartaut}).  The induction step follows from (\ref{eq:symmpolys}):
\begin{align*}
e_l&=\frac{1}{l}\sum_{j=1}^l (-1)^{j-1}\binom{d-k+l-j}{l-j}h^{l-j}\cdot (-1)^{j-1}(d-k+1)h^j\\
&=\frac{d-k+1}{l}\sum_{i=1}^l \binom{d-k+l-j}{l-j}h^l=\frac{d-k+1}{l} \binom{d-k+l}{l-1}h^l=\binom{d-k+l}{l}h^l,
\end{align*}
where in the third equality we used the hockey-stick identity.
\end{proof}
Thus, it follows that if $P$ is a monomial $c_{i_1}^{j_1}\cdots c_{i_l}^{j_l}$ of degree $k$ in Chern classes of $L^{[k]}$ on $C^{[k]}$ for a rational curve $C$, then $P_{0,d}=\binom{d-k+i_1}{i_i}\cdots\binom{d-k+i_l}{i_l}$.

\subsection{Counts of secant planes: a lucky case}\label{sec:lucky}
In this section we restrict ourselves to the case $e=4$, which, as mentioned before, has nice properties that are unfortunately not generalisable to higher values of $e$.  It does however offer an interesting application of Theorem \ref{thm:structure} and provides us also with some technical results needed in the proof of the general case in Section \ref{sec:generalcasecurves}.

Thus, the purpose of this section is to compute the integral
\[\int_{C^{[4]}} s_{2}(L^{[4]})^2 - s_{1}(L^{[4]})s_{3}(L^{[4]})\]
yielding the number of 4-secant lines to a curve of degree $d\geq 4$ in $\p^{3}$.
An elementary computation shows that 
\begin{equation}\label{eq:firstform}
s_{2}(L^{[4]})^2 - s_{1}(L^{[4]})s_{3}(L^{[4]}) = c_4(L^{[4]}) + s_4(L^{[4]}) - 6c_1(L^{[4]})ch_3(L^{[4]}).
\end{equation}
As discussed in the Introduction, the generating functions of the integrals
$\int_{C^{[4]}} c_4(L^{[4]})$ and $\int_{C^{[4]}}s_4(L^{[4]})$
have already been computed.  Using the Segre generating functions from Section \ref{sec:prelim} we can read off the coefficients to obtain (see for example Paragraph 5) in \autocite{LB}):
\begin{align*}
\int_{C^{[4]}} s_4(L^{[4]}) &= \binom{d}{4} - 3\binom{d}{3} + \binom{d}{2}(6-g) + 5d(g-2)+\binom{g}{2}-15g+15.
\end{align*}
To find a similar expression for $c_4(L^{[4]})$, one could read them off the coefficients of the generating functions in \autocite{Wa}, or one could use an elementary  recursive argument based on the multiplicativity of the total Chern class.  We show how to do this in the following

\begin{lemma}
Given a smooth curve $C$ with a line bundle $L$ of degree $d$.  We have the following closed form formulas for the top Chern and Segre classes of the tautological bundle $L^{[k]}$ on $C^{[k]}$:
\begin{align*}
\int_{C^{[4]}} c_4(L^{[4]})&= \binom{d+g}{4} - \sum_{m=0}^{g-1}\Biggl(\binom{d+g+m}{3} - \sum_{n=0}^{g-1}\binom{d+m+n}{2}\Biggr).
\end{align*} 
\end{lemma}
\begin{proof}
Let $C_1=C\sqcup C_2$, where $C_2$ is a smooth curve.  Equip $C_1$ with a line bundle $L_1$ and denote by $L$ and $L_2$ its restrictions to $C$ and $C_2$, respectively.  Let $p,p_2$ be the projections from $C^{[k_1]}\times C_2^{[k_2]}$ to $C^{[k_1]}$ and $C_2^{[k_2]}$, respectively. Then we have that 
\begin{align*}
 \int_{C_1^{[k]}} c(L_1^{[k]})&= \sum_{k_1+k_2=k}\int_{C^{[k_1]}\times C_2^{[k_2]}}c \bigl(p^* L^{[k_1]} \oplus p_2^* L_2^{[k_2]}\bigr) \\
 &=\sum_{k_1+k_2=k}\int_{C^{[k_1]}}c(L^{[k_1]})\int_{C_2^{[k_2]}}c(L_2^{[k_2]})\\
 &=\sum_{k_1+k_2=k}\int_{C^{[k_1]}}c_{k_1}(L^{[k_1]})\int_{C_2^{[k_2]}}c_{k_2}(L_2^{[k_2]}),
\end{align*}
where we used the fact that the only contribution to the integral comes from the respective top Chern classes.  We let now $C_2\simeq\p^1$ and $C$ be a smooth curve of genus $g$.  Moreover, let $L_2=\oo_{\p^1}(1)$ and let $L$ be a line bundle of degree $d$ on $C$ from which it follows that $\deg L_1=d+1$.  Furthermore, as explained in Section \ref{sec:reduction}, $g(C_1)=g(C)+g(C_2)-1=g-1$.  Denote the integral $ \int_{C^{[k]}} c(L^{[k]})$ by $C_{k,g,d}$.  We then rewrite the above equality as
\[ C_{k,g-1,d+1} = \sum_{k_1+k_2=k} C_{k_1,0,1} C_{k_2,g,d}. \]
We easily see that $C_{1,0,d}=d$ and it follows that
\begin{align*}
C_{1,g,d}&=C_{1,g-1,d+1} - C_{1,0,1}=C_{1,0,d+g} - gC_{1,0,1}=d.
\end{align*}
Furthermore,
\begin{align*}
C_{2,g,d}=C_{2,g-1,d+1} - C_{1,g,d}C_{1,0,1} - C_{2,0,1}.
\end{align*}
Since $C_{2,0,1}=0$, we obtain
\begin{align*}
C_{2,g,d}&=C_{2,g-1,d+1} - d=C_{2,0,d+g} - (d+(d+1)+\cdots+(d+g-1))\\
&=\binom{d+g}{2} - \biggl( dg + \frac{g(g-1)}{2}\biggr)=\frac{d(d-1)}{2}.
\end{align*}
We now write
\[C_{3,g,d}=C_{3,g-1,d+1} - C_{2,g,d}C_{1,0,1} - C_{1,g,d}C_{2,0,1} - C_{3,0,1}.\]
We have that $C_{3,0,1}=0$ which yields
\begin{align*}
C_{3,g,d} &= C_{3,g-1,d+1} - C_{2,g,d}=C_{3,g-1,d+1} - \frac{d(d-1)}{2}\\
&=C_{3,0,d+g} - \sum_{m=0}^{g-1}\binom{d+m}{2}\\
&=\binom{d+g}{3}- \sum_{m=0}^{g-1}\binom{d+m}{2}.
\end{align*}
We now finally reach the desired term:
\[ C_{4,g,d} = C_{4,g-1,d+1} - C_{3,g,d}C_{1,0,1} - C_{2,g,d}C_{2,0,1} - C_{1,g,d}C_{3,0,1} - C_{4,0,1}. \]
We again use that $C_{4,0,1}=0$ to obtain
\begin{align*}
C_{4,g,d}&=C_{4,g-1,d+1} - C_{3,g,d}=C_{4,0,d+g} - \sum_{m=0}^{g-1}C_{3,g,d+m}\\
&=\binom{d+g}{4} - \sum_{m=0}^{g-1}\Biggl(\binom{d+g+m}{3} - \sum_{n=0}^{g-1}\binom{d+m+n}{2}\Biggr).\qedhere
\end{align*}
\end{proof}

 The aim of the remainder of this section is therefore to find $\int_{C^{[4]}}c_1(L^{[4]})ch_3(L^{[4]})$ by using Theorem \ref{thm:structure}, or more precisely the proof thereof.  In fact we prove the following:

\begin{prop}\label{prop}
Let $C$ be a curve of genus $g$ equipped with a line bundle $L$ of degree $d$.  The closed formula 
\[\int_{C^{[k]}}c_1(L^{[k]})ch_{k-1}(L^{[k]})=\frac{(-1)^k}{(k-1)!}A(k,g,d),\]
where $A(k,g,d)=(d+g-k+1)^2 - g(2-k)(2-k-\frac{g-1}{2}) - (4-k)\frac{(g-2)(2d+g-1)}{2}$
holds for all $k\geq 1$ and its corresponding generating function is
\[ \sum_{k=0}^{\infty} z^k \int_{C^{[k]}}c_1(L^{[k]})ch_{k-1}(L^{[k]}) = (B(g,d)z^2+C(g,d)z+D(g,d))ze^{-z},  \]
where
\begin{align*}
B(g,d)&=(g-1)\Bigl(1-\frac{g}{2}\Bigr)+(g-2)(2d+g-1),\\
C(g,d)&=g-2d+2+(g-2)(2d+g-1),\\
D(g,d)&=-(d+g)^2+4g-\frac{g(g-1)}{2}.
\end{align*}
\end{prop}

\begin{proof}
To apply the setting of Theorem \ref{thm:structure} to our situation, we first remark on a slight change of notation intended to make computations easier to follow: from Theorem \ref{thm:egl} we know that, if $\phi$ is a polynomial function in Chern classes of line bundles $x$ on $C$, then the coefficients of $z^k$ in $H_{\phi}(C,x)$ depend only on $g$ and $d$.  We thus write $H_{\phi}(g,d)$ instead of $H_{\phi}(C,x)$.
We then take $\phi=c_1$ and $\psi = ch_{k-1}$ and set the following generating functions:

\begin{align*}
H_{\phi}(g,d) &:= \sum_{k=0}^{\infty} z^k \int_{C^{[k]}} c_1(L^{[k]}) = 1+dz\\
H_{\psi}(g,d) &:= \sum_ {k=1}^{\infty} z^k \int_{C^{[k-1]}} ch_{k-1}(L^{[k-1]})\\
H_{\phi\psi}(g,d) &:= \sum_{k=1}^{\infty} z^k \int_{C^{[k]}}c_1(L^{[k]})ch_{k-1}(L^{[k]}).
\end{align*}

We proceed by induction.  Consider the curve $C_1=C\sqcup C_2$ where $C_2\simeq\p^1$ and $C$ is a smooth curve of genus $g$.  Let $L_1$ be a line bundle on $C_1$ such that $(L_1)|_{C_2}=\oo_{\p^1}(d_2)$, for some positive integer $d_2$, and $L=(L_1)|_{C}$ has degree $d$.  As usual it follows that $\deg L_1=d+d_2$ and that $g(C_1)=g-1$.   From Theorem \ref{thm:structure} applied to the pair $(C_1,L_1)$ and setting $d_2=1$ we therefore obtain the following equality:
\begin{equation}\label{eq:recursion1}
H_{\phi\psi}(g-1,d+1)=H_{\phi\psi}(0,1) + H_{\phi\psi}(g,d) +H_{\psi}(0,1)H_{\phi}(g,d) + H_{\phi}(0,1)H_{\psi}(g,d).
\end{equation}
Thus the induction step decreases the genus of the curve and increases the degree of the line bundle at every step whence we obtain
\begin{equation}\label{eq:recursion2}
\begin{aligned}
H_{\phi\psi}(g,d)&=H_{\phi\psi}(g-1,d+1) - H_{\phi\psi}(0,1) - H_{\psi}(0,1)H_{\phi}(g,d) - H_{\phi}(0,1)H_{\psi}(g,d)\\
&= H_{\phi\psi}(0,d+g) - gH_{\phi\psi}(0,1) - H_{\psi}(0,1) \Sigma_1 - H_{\phi}(0,1)\Sigma_2,
\end{aligned}
\end{equation}
where 
\begin{align*}
\Sigma_1&=H_{\phi}(g,d)+H_{\phi}(g-1,d+1)+\cdots+H_{\phi}(1,d+g-1),\\
\Sigma_2&=H_{\psi}(g,d) + H_{\psi}(g-1,d+1) + \cdots + H_{\psi}(1,d+g-1).
\end{align*}
One sees immediately that $H_{\phi}(0,1)=1+z$ and moreover that
\begin{equation}\label{eq:sigma1}
\Sigma_1 = g+(d + (d+1) + \cdots + (d+g-1))z =g+ \frac{(g-2)(2d+g-1)}{2}z.
\end{equation}
We now compute the remaining unknown terms in (\ref{eq:recursion2}). 

\textbf{Step 1.}  Finding $H_{\psi}(g,d)$.

\noindent We consider first the auxiliary generating function 
\[\mathcal{F}_{g,d}(z) =\sum_{k=0}^{\infty} z^k \int_{C^{[k]}}ch_k(L^{[k]}) \]
which, by shifting the coefficients has the property that $z\mathcal{F}_{g,d}(z)=H_{\psi}(g,d)$.  Denote by $a_{k,g,d}$ the coefficient of $z^k$ in $\mathcal{F}_{g,d}(z)$.  From Lemma \ref{lemma} we have
\[a_{k,0,d}=\int_{\p^k} ch_k(\oo_{\p^1}(d)^{[k]}) = \frac{(-1)^{k-1}}{k!}(d-k+1),\]
 To find $a_{k,g,d}$ for $g\neq 0$, we again take advantage of the additivity of the Chern character as follows:  as before let $C_1=C\sqcup C_2$ be the disjoint union of two smooth curves and $L,L_1,L_2$ and $p,p_2$ as before.  Then, for all $k\geq 0$:
\begin{align}\label{eq:cherncharcomp}
a_{k,g(C_1),\deg(L_1)} &= \int_{C_1^{[k]}}ch_k (L_1^{[k]}) =  \sum_{k_1 + k_2 = k} \int_{C^{[k_1]}\times C_2^{[k_2]}} ch_k \bigl(p^* L^{[k_1]} \oplus p_2^* L_2^{[k_2]}\bigr)\nonumber\\
&=\sum_{k_1 + k_2 =k}\int_{C^{[k_1]}}ch_k \bigl(L^{[k_1]}\bigr) + \int_{C_2^{[k_2]}}ch_k \bigl(L_2^{[k_2]}\bigr)\nonumber\\
&= \int_{C^{[k]}} ch_k \bigl(L^{[k]}\bigr) + \int_{C_2^{[k]}} ch_k \bigl(L_2^{[k]}\bigr).
\end{align}
Setting as usual $C_2=\p^1$ and $L_2=\oo_{\p^1}(1)$ and taking $C$ to be a smooth curve of genus $g$ and $L$ a line bundle of degree $d$ and plugging into (\ref{eq:cherncharcomp}) yields:
\[ a_{k,g,d} = a_{k,g-1,d+1}-a_{k,0,1}. \]
Using induction on $g$ we obtain
\begin{equation}\label{eq:chernchar}
a_{k,g,d}=a_{k,0,d+g} - ga_{k,0,1}=\frac{(-1)^{k-1}}{k!}(d+(k-1)g-k+1).
\end{equation}
To find the generating function $\mathcal{F}_{g,d}(z)=\sum_{k=0}a_{k,g,d} z^k$, we start from the recurrence relation for the coefficients $a_{k,g,d}$:
\[ -(k+1)a_{k+1,g,d}=a_{k,g,d} + \frac{(-1)^{k-1}}{k!}(g-1). \]
This recurrence relation yields the following differential equation for $\mathcal{F}_{g,d}(z)$:
\[\frac{d}{dz}\mathcal{F}_{g,d}(z) + \mathcal{F}_{g,d}(z)=(g-1)e^z,\]
which, using the initial condition $a_{0,g,d}=-(d-g+1)$, we can easily solve to obtain:
\[ \mathcal{F}_{g,d}(z)=((g-1)z-d+g-1)e^{-z}. \]
Finally, we have that
\[H_{\psi}(g,d)=((g-1)z-d+g-1)ze^{-z}.\]
Note that this gives an alternative way of obtaining the generating function for the Chern character of tautological bundles on symmetric products for curves.  For a different approach using Grothendieck-Riemann-Roch, see \autocite[Lemma 2.5, Chapter VIII \S 2]{ACGH}.

\textbf{Step 2.} Finding $\Sigma_2$.

\noindent Having computed $H_{\psi}(g,d)$, we now easily get
\begin{align*}
\Sigma_2 &= H_{\psi}(g,d) + H_{\psi}(g-1,d+1) + \cdots + H_{\psi}(1,d+g-1)\\
&=(((g-1)z-d+g-1)+\cdots+(-d-g+1))ze^{-z}\\
&=\biggl(\frac{g(g-1)}{2}z-\Sigma_1+\frac{g(g-1)}{2}\biggr)ze^{-z}.
\end{align*}

\textbf{Step 3.}  Finding $H_{\phi\psi}(0,d+g)$ and $H_{\phi\psi}(0,1)$.

\noindent   We use the same inductive strategy as in Step 1.  Note that the coefficient of $z^k$ in $H_{\phi\psi}(0,d)$ is given by
\begin{align*}
\int_{\p^k} c_1(\oo_{\p^1}(d)^{[k]})ch_{k-1}(\oo_{\p^1}(d)^{[k]})=\frac{(-1)^{k}}{(k-1)!}(d-k+1)^2,
\end{align*}
where we again used Lemma \ref{lemma}.
We shall compute the auxiliary generating function
\[ \mathcal{G}_d(z) = \sum_{k=0}^{\infty}b_{k,d} z^k :=\sum_{k=0}^{\infty} \frac{(-1)^{k+1}}{k!}(d-k)^2 z^k \]
with the property that $H_{\phi\psi}(0,d)=z\mathcal{G}_d(z)$.  To do so, we observe that the recurrence relation 
\[ -(k+1)b_{k+1,d} = b_{k,d} + \frac{(-1)^k}{k!}(2d-1) +2 \frac{(-1)^{k-1}}{(k-1)!} \]
between the coefficients of $z^k$ in $\mathcal{G}_d(z)$ gives rise to the differential equation
\[ \frac{d}{dz}\mathcal{G}_d(z) + \mathcal{G}_d(z) + (2d-1)e^{-z} + ze^{-z} =0.\]
With the initial condition $b_0=-d^2$ we find that the solution is
\[ \mathcal{G}_d(z) = -\biggl(\frac{z^2}{2}+(2d-1)z+d^2\biggr)e^{-z}, \]
so that we finally obtain
\[ H_{\phi\psi}(0,d) = -\biggl(\frac{z^2}{2}+(2d-1)z+d^2\biggr)ze^{-z}. \]
Putting everything together yields the expression of $H_{\phi\psi}(g,d)$ in the statement of the Proposition.

To get the closed form of the coefficient of $z^k$ in $H_{\phi\psi}(g,d)$, one could either read it off from the generating function itself or just apply the same induction procedure: for a disjoint union $C_1=C\sqcup C_2$ of smooth curves equipped with the usual line bundles $L,L_1,L_2$, the following holds:
\begin{align*}
\int_{C'^{[k]}} c_1(L'^{[k]})ch_{k-1}(L'^{[k]})&=\int_{C_1^{[k]}}c_1(L_1^{[k]})ch_{k-1}(L_1^{[k]})+\int_{C_1}c_1(L_1)\int_{C_2^{[k-1]}}ch_{k-1}(L_2^{[k-1]})+\\
&+\int_{C_1^{[k-1]}}ch_{k-1}(L_1^{[k-1]})\int_{C_2}c_1(L_2)+\int_{C_2^{[k]}}c_1(L_2^{[k]})ch_{k-1}(L_2^{[k]}).
\end{align*}
Let
\[b_{k,g,d}:=\int_{C^{[k]}}c_1(L^{[k]})ch_{k-1}(L^{[k]}),\] 
With the same choices for $C_1$, $C_2$, $L_1$, $L_2$ as above we have that
\[ b_{k,g-1,d+1} = b_{k,0,1} + 1\cdot a_{k-1,g,d} + a_{k-1,0,1}\cdot d + b_{k,g,d}. \]
This in turn yields
\begin{align*}
b_{k,g,d}&=b_{k,g-1,d+1}-\frac{(-1)^k}{(k-1)!}(2-k)^2 - \frac{(-1)^k}{(k-1)!}(d+(k-2)g-k+2)-d\frac{(-1)^{k}}{(k-1)!}(3-k)\\
&=b_{k,g-1,d+1}-\frac{(-1)^{k}}{(k-1)!}((2-k)^2 + (4-k)d + (k-2)(g-1))\\
&=b_{k,0,d+g} - \frac{(-1)^{k}}{(k-1)!}\biggl(g(2-k)^2 + (4-k)\frac{(g-2)(2d+g-1)}{2} + (k-2)\frac{g(g-1)}{2}\biggr).
\end{align*}
Finally, we have that
\[ b_{k,0,d+g}= \frac{(-1)^k}{(k-1)!}(d+g-k+1)^2,\]
which concludes the proof. 
\end{proof}

We therefore have
\begin{align*}
\int_{C^{[4]}}c_1(L^{[4]})ch_3(L^{[4]})= b_{4,d,g}=\frac{1}{6}((d+g-3)^2-4g-g(g-1)).
\end{align*}

\subsection{Counts of secant planes: the general $f=2$ case}\label{sec:generalcasecurves}
Unfortunately we were not able to find a similar expression to the one in (\ref{eq:firstform}) for higher values of $e$.  Therefore, to deal with the general situation we take a different approach, although we will still rely on a similar induction step and Theorem \ref{thm:structure} for the final computation.  

Thus, the goal of this section is to find an expression for the integral
\[ \int_{C^{2k}} s_{k}^2(L^{[k]}) - s_{k-1}(L^{[k]})s_{k+1}(L^{[k]}), \]
where $C$ and $L$ are as before and $k\geq 2$.  To proceed, recall that the total Segre class of the dual vector bundle $(L^{[k]})^{\vee}$ may be written in terms of the complete homogeneous symmetric polynomials $h_i$ in the Chern roots as follows:
\[s((L^{[k]})^{\vee}) = 1 + h_1(\alpha_1,\ldots,\alpha_k) +\cdots + h_k(\alpha_1,\ldots,\alpha_k),\]
and the Segre power series $s_t=1+s_1 t + s_2 t^2 + \cdots $ can therefore be written as:
\[ s_t((L^{[k]})^{\vee}) = \prod_{i=1}^{\infty} \frac{1}{1-\alpha_i t} \]
We now obtain the generating function with coefficients given by $s_i^2(L^{[k]})=s_i^2((L^{[k]})^{\vee})$.   To do so, note that for two generating functions $F(z)=\sum_{i=1}^{\infty} a_i z^i$ and $G(z)=\sum_{i=1}^{\infty} b_i z^i$, their Hadamard product is given by
\[ (F*G)(z) = \sum_{i=1}^{\infty} a_i b_i z^i. \]
\begin{lemma}
	(i) For $k\geq 1$ we have:
	 \[s_t((L^{[k]})^{\vee})*s_t((L^{[k]})^{\vee}) = \exp\left( \sum_{n\geq 1}\frac{t^n}{n}(n!)^2 ch_n^2(L^{[k]}) \right).\]
	 Moreover, the coefficient $s_n^2((L^{[k]})^{\vee})$ of $t^n$ is given by
	 \begin{equation}\label{eq:coeff1}
	 \sum_{n_1,n_2,\ldots\geq 0}  \frac{1}{n_1 ! n_2 ! \cdots} \frac{(1! ch_1^2)^{n_1} (2! ch_2^2)^{n_2}\cdots}{1^{n_1} 2^{n_2}\cdots},
	 \end{equation}
	 where $n_1 + 2n_2 + \cdots = n$.
	 
	(ii) Denote by $s_t^l$ and $s_t^r$ the generating functions obtained by shifting $s_t$ by one to the left and to right, respectively, i.e.
	\begin{align*}
	s_t^l &= t + s_1 t^2 + \cdots + s_{k-1}t^k + \cdots,\\
	s_t^r &= s_1 + s_2 t + \cdots + s_{k+1}t^k + \cdots.
	\end{align*}
	 Then the coefficient $s_{n-1}(L^{[k]})^{\vee})s_{n+1}(L^{[k]})^{\vee})$ of $t^n$ in $s_t^l*s_t^r$ is
	\begin{equation}\label{eq:coeff2}
	\left(\sum_{n_1,n_2,\ldots\geq 0}  \frac{1}{n_1 ! n_2 ! \cdots} \frac{(1! ch_1)^{n_1} (2! ch_2)^{n_2}\cdots}{1^{n_1} 2^{n_2}\cdots}\right)\left(\sum_{l_1,l_2,\ldots\geq 0}\frac{1}{l_1 ! l_2 ! \cdots }\frac{(1! ch_1)^{l_1} (2! ch_2)^{l_2}\cdots}{1^{l_1} 2^{l_2}\cdots}\right)
	\end{equation}
	where $n_1 + 2n_2 + \cdots = n-1$ and $l_1 + 2l_2 + \cdots = n+1$
\end{lemma}
\begin{proof}
(i) This follows by standard results on symmetric functions.  In particular, one has that
\[s_t * s_t = \prod_{i,j=1}^k \frac{1}{1-\alpha_i \alpha_j t},\]
which then implies that
\[ \sum_{n\geq 1}\frac{t^n}{n} \sum_{i,j=1}^k (\alpha_i \alpha_j)^n = \ln (s_t*s_t). \]
Using the fact that $ch_n = \frac{p_n}{n!}$ we get  $\sum_{i,j=1}^k (\alpha_i \alpha_j)^n =p_n^2= (n!)^2 ch_n^2$, where $p_n$ denotes as before the $n$-th power sum symmetric polynomial in the Chern roots of $L^{[k]}$.  The expression of the coefficient of $t^n$ follows immediately.

(ii) This follows by a standard combinatorial computation and recognising that $s_t^l=ts_t$ and $s_t^r=\frac{1}{t}(s_t-1)$.
\end{proof}

Thus, to understand the structure of the formula for 
\[\int_{C^{[2k]}} s_k^2(L^{[k]}) - s_{k-1}(L^{[k]})s_{k+1}(L^{[k]}),\]
it is enough to make sense of the integrals of the expressions (\ref{eq:coeff1}) and (\ref{eq:coeff2}).  To do so we use Corollary \ref{cor:struct} to find separate expressions for 
\begin{equation}\label{eq:bla1}
\int_{C^{[2k]}} s_k^2(L^{[k]})
\end{equation} 
and
\begin{equation}\label{eq:bla2}
\int_{C^{[2k]}} s_{k-1}(L^{[k]})s_{k+1}(L^{[k]}).
\end{equation}
To obtain (\ref{eq:bla1}) we must therefore consider the integral
\[ \int_{C^{[2k]}} \left(\sum_{n_1,n_2,\ldots\geq 0}C_{n_1,n_2,\ldots} ch_1^{2n_1}(L^{[k]})ch_2^{2n_2}(L^{[k]})\cdots\right)  \]
where $n_1 + 2n_2 + \cdots = k$ and $C_{n_1,n_2,\ldots}$ denote the constants from (\ref{eq:coeff1}).  Now, for each choice of $n_i$ with $n_1+2n_2+\cdots = k$, Corollary \ref{cor:struct} yields
\begin{align*}
\int_{C^{[2k]}}ch_1^{2n_1}(L^{[k]})ch_2^{2n_2}(L^{[k]})\cdots =\mathcal{L}+\prod_{i}\Bigl(\int_{C^{[i]}}ch_i(L^{[i]})\Bigr)^{2n_i}+\\
 + \sum_{\substack{0\leq i_1\leq 2n_1\\ 0\leq i_2\leq 2n_2\\ \ldots}}\mathcal{L}_{i_1\ldots i_j}\Bigl(\int_{C}ch_1(L)\Bigr)^{i_1} \Bigl(\int_{C^{[2]}}ch_2(L^{[2]})\Bigr)^{i_2}\cdots,
\end{align*} 
where  $\mathcal{L}$ and the $\mathcal{L}_{i_1\ldots i_j}$ are linear functions as in the context of Corollary \ref{cor:struct}.  Note that for simplicity we omit the scalar multiples before each term in the sum.  Finally, each term between parantheses in the above equality can be found by either reading off the coefficients of the generating function $\mathcal{F}_{g,d}$ from Step 1 of the proof of Proposition \ref{prop}, or directly by a recursive argument as done in the case of the coefficients $b_{k,g,d}$ from loc.cit. With either method one obtains:
\[ \int_{C^{[i]}}ch_i(L^{[i]}) =  \frac{(-1)^{i-1}}{i!}(d+(1-i)(1-g)).\]
This then yields the desired structure result for (\ref{eq:bla1}).  A similar expression can then be obtained for (\ref{eq:bla2}).


\printbibliography{}
\end{document}